\documentclass[11pt,twoside]{article}
\usepackage{latexsym}
\usepackage{amssymb,amsbsy,amsmath,amsfonts,amssymb,amscd}
\usepackage{graphicx}
\usepackage{color}
\usepackage{hyperref}

\newcommand{\commentout}[1]{}

\newcommand {\e}  {\varepsilon}

\newcommand {\Chi} {{\bf \raise 2pt \hbox{$\chi$}} }
\newcommand {\f}   {\frac}
\newcommand {\p}   {\partial}
\newcommand{\dis}{\displaystyle}

\newcommand{\beq}{\begin{equation}}
\newcommand{\eeq}{\end{equation}}
\newcommand{\bea} {\begin{array}{rl}}
\newcommand{\eea} {\end{array}}
\newcommand{\bepa}{\left\{ \begin{array}{l}}
\newcommand{\eepa} {\end{array}\right.}
\newtheorem{theorem}{Theorem}[section]
\newtheorem{lemma}[theorem]{Lemma}

\newtheorem{proposition}[theorem]{Proposition}

\newcommand{\qed}{{ \hfill
                       {\unskip\kern 6pt\penalty 500 \raise -2pt\hbox{\vrule\vbox to 6pt{\hrule width 6pt
                       \vfill\hrule}\vrule} \par}   }}
\title{  Global solution for a kinetic chemotaxis model   with internal dynamics and its fast adaptation limit  }
\author{%Author  \and
Jie Liao\thanks{Department of Mathematics, East China University of Science and Technology,  Shanghai, 200237, P. R. China, Email: liaojie@ecust.edu.cn.  
  }
}

\date{\today}

\begin{document}
\maketitle
\pagestyle{plain}
%\tableofcontents
\pagenumbering{arabic}

\numberwithin{equation}{section}

\begin{abstract}

 A  nonlinear kinetic chemotaxis model  with internal dynamics  incorporating signal transduction and adaptation is considered. 
This paper is concerned with: (i) the  global solution for this model,  and, (ii)  its fast adaptation limit to Othmer-Dunbar-Alt type model.
 This limit  gives some  insight to the molecular origin of  the chemotaxis behaviour.

 First, by using the Schauder fixed point theorem, the global existence of weak solution is proved based on detailed a priori estimates, 
 under some quite general assumptions on the model and the initial data. 
However, the Schauder fixed point theorem does not provide uniqueness. Therefore, additional analysis is required to be developed to obtain uniqueness. 
% {\color{red}    
% The regularity and uniqueness of the global weak solution are also considered.
 %}
 
Next, the fast adaptation limit of this model  is derived by extracting a weak convergence subsequence in measure space. For this limit,
the first difficulty is to show the concentration effect on  the internal state. 
When the small parameter $\e$, the adaptation time scale, goes to zero, we prove that the solution converges to a Dirac mass in the internal state variable. Another difficulty is the strong compactness argument on the chemical potential, which is essential for passing the nonlinear kinetic equation to the weak limit.

\end{abstract}

\bigskip

\noindent {\bf Key words:}  kinetic chemotaxis model; internal dynamics; global solution; fast adaptation limit
%Asymptotic analysis; 
\\
\noindent {\bf Mathematics Subject Classification (2010):}  35B25; 35Q92;   80A30

%t\newpage

%%%%%%%%%%%%%%%%%%%%%%%%%%%%%%%%%%%%%%%%%%%%
\section{Introduction}
\label{sec:intro}
%-------------------------------------------
%%%%%%%%%%%%%%%%%%%%%%%%%%%%%%%%%%%%%%%%%%%%

% Chemotaxis models are of major importance in many fundamental processes in biology and physiology, since chemotaxis  plays an important role for numerous biological systems such as embryogenesis, wound healing, immune response,    tumor biology and more.  
 
%  The chemical signal can be secreted by the species itself or supplied to it by external source

Chemotaxis is a mechanism by which cells or bacteria efficiently and rapidly respond to changes in the chemical composition of their environment, for example approaching chemically favorable environments or avoiding unfavorable ones. 
This behavior is achieved by two major steps: (i) detection of the signal and (ii) integrate the signals received from receptors that triggers the response. We consider the  kinetic chemotaxis model  with internal dynamics  incorporating signal transduction and adaptation 
 \beq \label{intr1}
 \left\{
\begin{array}{l}
 \p_t p  +  v \cdot \nabla_x p  + \p_m [ { F(m, S ) \over \e}  p ] = {\dis \int_V }T(v,v',m) p (x,v',m,t) - T(v',v,m) p (x,v,m,t) dv', \vspace{2mm}\\
 -\Delta S  +S  = n (x,t):= {\dis \int_{0}^{\infty}\int_{V} }p (x,v,m,t) dv dm, \vspace{2mm} \\
 p (x,v, m=0,t)=0, \qquad p (x,v, m=+\infty,t)=0,
\end{array}
\right.
 \eeq
where $p (x,v,m,t)$ denotes the cell density at position $x\in \mathbb{R}^d$, velocity $v\in V$,  with internal state $m>0$ (for example, methylation level) and time $t$, $S(x,t) $ is the  chemical potential.  
The velocity space $V$ is assumed to be a   compact domain in $\mathbb{R}^d$, typical example being unit ball $\{|v|\leq 1\}$. 
The  parameter $\e$ in the model characterises the adaptation time scale.

This model has been developed in \cite{SPO,BL,M}, and further elaborated by  \cite{EO,EO1}, to connect  the aspects of the signal transduction and response to the macroscopic equations  by using moment closure techniques.  
This type of model was also studied in some other references, for example \cite{STY, Si,Tu}, to bridge the molecular level pathway dynamics with the  cellular behavior such as  cell population level motility  in some biological systems.
 These works made possible the development of predictive agent-based models that include the intracellular signaling pathway dynamics. For instance, it is of great biological interest to understand the molecular origins of chemotactic behavior of E. coli by deriving a population-level model based on the underlying signaling pathway dynamics. 
 We further mention that, this model also has close relation to the kinetic theory for active particles \cite{B1,B2} (KTAP), which has been applied to model various complex systems in life sciences, for instance biological growing tissues \cite{B3,JanK}, social systems \cite{[5]}, behavioural economy \cite{[6]} or epidemics with gene mutations \cite{[7]}.

The mathematical study of chemotaxis dated back to as early as the work of Patlak \cite{Patlak} in 1953,  and  further Keller and Segel  derived one of the best studied models in mathematical biology   at the macroscopic  population level \cite{KS,KS1,KS2}. We refer \cite{HP} and references therein on this model. The famous Keller-Segel model has successfully  explained chemotactic phenomena in slowing changing environments \cite{Tin}. On the other hand,  
in order to understand bacterial behavior at the individual level, the Othmer-Dunbar-Alt kinetic model  was proposed by Alt  (1980) \cite{Alt} and Othmer et al. (1988) \cite{ODA} for the description of the chemotactic movement of cells in the presence of a chemical substance, and the Keller-Segel model can be derived by taking the hydrodynamic limit of kinetic models (see  \cite{Cha,DS,Fil,JV} and references therein).  In the conventional Othmer-Dunbar-Alt kinetic models,  a (biased) velocity jump assumption in the turning kernel is always assumed.  \\

%transduction of the external signal into an internal signal that triggers the response. 

%integrate the signals received from receptors and modulating the direction of flagellar rotation accordingly.

%To incorporate further the underlying internal signal transduction and adaptation process,      pathway dynamics.  
 
%especially to  integrate signals received from receptors that sense the environment and modulating the direction of flagellar rotation accordingly.  
  
%  The motivation is that it should explain, by means of a physical process,  the equation for the run-tumble phenomena that people as Nicolas and Casimir are using (either delay either depending on $v. \nabla S$

%Cells move by a series of 'run and tumble' corresponding to the clockwise or counterclockwise activations of their flagellas  in response to chemoattractant substances and receptors saturation.

This paper is concerned with the study of the kinetic chemotaxis model with internal dynamics \eqref{intr1}. Our first goal is to prove existence of global solutions for the Cauchy problem of \eqref{intr1} under some quite general assumptions. 
We   assume the initial condition
\beq
p_{0}(x,v,m)\geq 0, \quad  (1+ \langle x \rangle )p_{0}(x,v,m) \in L^{1} (\mathbb{R}^d_x \times V \times \mathbb{R}^{+}_m), 
 \quad p_{0}(x,v,m) \in L^{\infty} (\mathbb{R}^d_x \times V \times \mathbb{R}^{+}_m),
\label{as:in1}
\eeq 
where $\langle x \rangle = \sqrt{1+x^2}$, and
\beq
\bar p_{0}(x,v) := \int_0^\infty p_{0}(x,v,m) dm  \in L^{\infty}(\mathbb{R}^d_x \times V).
\label{as:in2}
\eeq
We will use the notation
$$
\bar{p} (x,v,t) := \int_{0}^{\infty} p  (x,v,m,t)  dm .
$$
 
 %For the global existence of \eqref{frl1}-\eqref{frl2},        % we refer the result of Erban-Hwang\cite{EH}, and 
The assumptions on $F(m,S)$ and $T(v,v',m)$ are:
\begin{itemize}
\item Assumption on the adaptation rate $F$. There  exists a non-negative, non-decreasing continuous function $\Pi(\cdot)\in C(\mathbb{R})$ %and a positive constant $C_{\mathcal{F}} $ 
such that
\beq\label{frl8}
%| F(m,S)| +
 | \p_{m}   F(m,S)| \leq   \Pi(S).       %     C_{\mathcal{F}}(1+ \Pi(S)).
\eeq
\item Assumption   on the turning kernel $T$. The turning kernel $T$ is positive and uniformly bounded: there exist a positive constant $C_{\mathcal{T}}$ such that 
\beq\label{as:T}
0< T(v,v',m) \leq C_{\mathcal{T}}. % (1+ \Lambda(m)),
\eeq
%where $\Lambda \in C(\mathbb{R^+})$ is a non-negative, non-decreasing continuous function. 
\end{itemize}
Under the above assumptions, we prove the global existence of weak solution to the system \eqref{intr1}, see Theorem \ref{global} in Section \ref{section2}.
% {\color{red}   
However, the Schauder fixed point theorem does not provide uniqueness.  
With further assumptions on the coefficients and initial data,  the regularity and uniqueness of the global weak solution are also considered in Section \ref{reguni}. \\
 %}

The second goal is to study the fast adaptation limit as $\e \to 0 $ of  \eqref{intr1}. For the reasons below (compatible with \cite{STY,Si}), we make  further assumption on $F$ that
\beq \label{as:F1}
\begin{array}{rl}
&\text{There exist positive constants } m_{\pm} \text{ and } m_0(S ),  m_{-}< m_0(S ) <  m_+ ,  \text{  such that }\vspace{2mm} \\
&F(m, S ) > 0 \quad \text{for } \; m < m_0(S ), \qquad   F(m, S ) < 0 \quad \text{for } \; m > m_0(S ).
\end{array}
\eeq
And to simplify the proof,  we also assume the initial data does not   concentrate on large $m$, i.e.,
\beq\label{mar11-4}
  \int_{0}^{\infty}\int_{\mathbb{R}^d_x} \int_V  m p_0 dv dx dm  < \infty.
\eeq
Under the above assumptions, we  study in Section \ref{section3}, as $\e \to 0 $, how the kinetic chemotaxis model with internal dynamics \eqref{intr1} is close to the Othmer-Dunbar-Alt type kinetic model
$$
 \left\{
\begin{array}{rcl}
\p_t \bar{p} +   v\cdot\nabla_x\bar{p} 
&=&  {\dis \int_V } T(v,v',m_0(S))  \bar{p}(x,v',t) -  T(v',v,m_0(S))   \bar{p}(x,v,t) dv' ,  \\
 -\Delta S  +S &=& n (x,t):=  {\dis \int_{V}} \bar p (x,v,t) dv  .
\end{array}
\right.
$$
 The resulting limit, Theorem \ref{f13-1},   gives some  insight to the molecular origins of the  chemotaxis behaviour.\\

% {\color{red}    
We conclude this introduction by mentioning  that, other types of scaling and related limits can be considered.   For example, based on both fast adaptation and stiff response, the paper \cite{PTV} studied how the path-wise gradient of chemotactic signal arises from intra-cellular molecular content. See further remarks in the conclusion section.
%}

%%%%%%%%%%%%%%%%
%----------------------------------------
%\section{Hyperbolic scaling}
%---------------------------------------
%%%%%%%%%%%%%%%%

%====================================

%\newpage

%%%%%%%%%%%%%%%%%%%%%%
%-------------------------------------------------------
\section{Global existence of  solutions}\label{section2}
%-------------------------------------------------------
%%%%%%%%%%%%%%%%%%%%%%

Existence of solutions for this type of problem has been considered in \cite{EH}. See also \cite{BC}. Here, we extend the results under more general assumptions. We also give uniform bounds which will be useful for later purpose. 
The parameter $\e$ does not play a role here thus we take $\e=1$ for simplicity and rewrite the original equation for $p$ as
\beq \label{frl1}
 \p_t p  +  v \cdot \nabla_x p  + \p_m [  F(m, S )   p ] = \int_V T(v,v',m) p (x,v',m,t) - T(v',v,m) p (x,v,m,t) dv',
\eeq
coupled with the steady elliptic equation for the  chemical signal $S $
\beq \label{frl2}
-\Delta S  +S  = n (x,t):= \int_{0}^{\infty}\int_{V} p (x,v,m,t) dv dm,
\eeq
and the boundary constraint 
\beq \label{frl3}
 p (x,v, m=0,t)=0, \qquad p (x,v, m=+\infty,t)=0.
\eeq
%Above,  $V$ is a   compact set in $\mathbb{R}^{d}$,  and for example we can let $V$ be the unit ball in $\R^d$ for simplicity.

Note that from Equation \eqref{frl2}, we have
 $$
 S (x,t) = G*n (x,t),  $$  
where $G$ is the Bessel potential. For later use, we recall some properties of the Bessel potential in below.

\begin{proposition}\label{BP}
 (Properties of Bessel potential  \cite{S}) 

(P1) $G(x)\geq 0$, $\|G\|_{L^1(\mathbb{R}^d)} = 1$,

(P2)  For any spatial dimension $d$, 
 $$ G(x) \in L^q(\mathbb{R}^d), \quad \forall 1\leq q< {d \over d-2},        
 \quad (\text{we understand }{d \over d-2}= \infty \text{ when }d=1,2),  $$
 
 (P3)  For any $\alpha\in \mathbb{R}$, $\beta>0$, $$\int_{|x|>1}   |x|^{\alpha}  G^\beta(x)  dx<\infty, $$

(P4)  For any spatial dimension $d$,  
  $$\nabla G(x) \in  L^1(\mathbb{R}^d).$$
 
\end{proposition}

  We briefly mention that, 
for $d=1$, $ G(x) = \frac{1}{2} e^{-|x|}$. % , we directly have  $ G(x) \in L^q(\mathbb{R})$ for all $1\leq q \leq \infty$. 
For $d=2$, $G(x)$  has the singularity of $log|x|$ at $|x|\rightarrow 0$. 
In the case  $d\geq 3$,  the only singularity of $G(x)$ at $x=0$ behaves  as  
$$
G(x) \sim  |x|^{-(d-2)} \quad \text{as}\quad |x|\rightarrow 0.
$$
Also note that $G(x)$ decreases exponentially as $|x|\rightarrow \infty$,  therefore we   obtain  $(P2)$-$(P4)$.

\subsection{Statement of the result}

%-----------------------------------------------
\begin{theorem}\label{global} (Global existence)
Under the assumption \eqref{frl8} on $F$,  \eqref{as:T} on $T$, consider system \eqref{frl1}-\eqref{frl3} with initial data $p_0$ satisfying  \eqref{as:in1}-\eqref{as:in2},  there exists a global weak  solution $(p, S)$ that   
\beq\label{prop:ps}
\|  S (\cdot,t) \|_{L^{1}\cap L^{\infty} (\mathbb{R}^d_x)} \leq  \|p_{0}\|_{L^{1}} +
  V_d  \|\bar p_{0}\|_{L^{\infty}} (1+2V_dC_{\mathcal{T}}t e^{2V_dC_{\mathcal{T}}t}), \quad \forall t>0,
\eeq
\beq\label{prop:p1}
\|  p (\cdot,\cdot,\cdot,t) \|_{L^{1}\cap L^{\infty} (\mathbb{R}^d_x \times V \times \mathbb{R}^{+}_m)} \leq  
\|p_{0}\|_{L^{1}} +
\|p_{0}\|_{  L^{\infty}}  \big( 1+ \tilde{C} t e^{\tilde{C} t}  \big), \quad \forall t>0,   %\quad \forall 1\leq q \leq \infty,
\eeq
where $V_d$ is the volume of   $V$, $\tilde{C}>0$ is a constant depends  on the estimate of $S$, %$C_{\mathcal{F}} $, 
$C_{\mathcal{T}}$ and initial data.
Moreover, % $\bar{p} $ satisfies  
\beq\label{prop:p2}
\|  \bar{p} (\cdot,\cdot,t) \|_{L^{\infty} (\mathbb{R}^d_x \times V  )} \leq  \|\bar p_{0}\|_{L^{\infty}}  (1+2V_dC_{\mathcal{T}}t e^{2V_dC_{\mathcal{T}}t}),\quad \forall t>0,
\eeq
\beq\label{nlinf}
\| n (\cdot,t) \|_{L^{\infty} (\mathbb{R}^d_x  )} \leq V_d \|\bar p_{0}\|_{L^{\infty}}  (1+2V_dC_{\mathcal{T}}t e^{2V_dC_{\mathcal{T}}t}),\quad \forall t>0.
\eeq

\end{theorem}

We recall that the global existence result was considered in \cite{EH}, with more precise examples on $T$ and $F$. They take the turning kernel $T$ as a product of the turning frequency $\lambda$ and the kernel $K$, i.e.,
$$
T(v,v',m) = \lambda(m)K(v,v',m), 
$$
in which the kernel $K$ is non-negative and satisfies the normalisation condition
$$
\int_{V} K(v,v',m)dv=1,
$$
where $V$ is a symmetric compact set in $\mathbb{R}^{d}$. The simplest  example of $K$ is 
$$
 K(v,v',m) = {1\over V_d},
$$
%where $V_d$ is the volume of $V$,  
to assume constant turning probability of changing velocity from $v'$ to $v$, or 
$$
 K(v,v',m) = k{\theta}, \quad cos(\theta)= {v\cdot v' \over |v| |v'|},
$$
to assume that the turing kernel is a function of the angle between original and new velocities. More general, one can assume that it is uniformly bounded by a constant , i.e.,
$$%\beq\label{frl5}
| K(v,v',m)| \leq C.
$$% \eeq
For the (output) turning frequency $\lambda(m)$, which is related to the  (input) signal function seen by a cell, to prevent formation of singularities, the authors in \cite{EH} suppose the growth estimate 
$$%\beq\label{frl6}
\lambda(m) \leq C\Big(1+ \Lambda(S) + | {\p S \over \p t} +  v\cdot {\p S\over \p x}|\Big), \qquad \forall x, \; v, \; t,
$$%\eeq
where $\Lambda \in C(\mathbb{R^+})$ is a non-negative, non-decreasing continuous function. An assumption that, as we see it below, is not necessary for our mathematical treatment. 
 \bigskip

\subsection{A priori estimates}

  \begin{itemize}
\item  $L^1$ bounds on $p$, $\bar p$ and $n$.
\end{itemize}

 Take the integration of \eqref{frl1} with respect to $x,v,m$ variables, we have the mass conservation
$$
{d\over dt}\int_{\mathbb{R}^d_x}\int_{0}^{\infty}\int_{V}  p (x,v,m,t)  dv dm  dx\equiv 0,
$$
which is also
$$
\| p (\cdot, \cdot,\cdot, t) \|_{L^{1}(\mathbb{R}^d_x \times V\times \mathbb{R}^+_m )}
 =  \| \bar p (\cdot,\cdot, t) \|_{L^{1}(\mathbb{R}^d_x\times V)}
= \| n (\cdot, t) \|_{L^{1}(\mathbb{R}^d_x)} 
$$
\beq\label{nl1}
\equiv   \int_{\mathbb{R}^d_x}
 \int_{0}^{\infty}\int_{V} p_{0}(x,v,m) dv dmdx = \|p_{0}\|_{L^{1}}.
 %(\mathbb{R}^d_x \times V \times \mathbb{R}^+_m)}.
\eeq

  \begin{itemize}
\item  $L^\infty$ bounds on $\bar p$ and $n$.
\end{itemize}

On the other hand, take the integration of \eqref{frl1} only with respect to $m$ over $[0,\infty)$ and use the boundary condition \eqref{frl3}, we have
$$
 \p_t \bar p  +  v \cdot \nabla_x \bar p   = \int_{0}^{\infty} \int_V T(v,v',m) p (x,v',m,t) - T(v',v,m) p (x,v,m,t) dv' dm,
$$
then it can be further represented by 
$$
\begin{array}{rcl}
  \bar p   (x,v,t) & =&  \displaystyle \bar p _{0}  (x-vt,v)    + \int_{0}^{t} \int_{0}^{\infty} \int_V  \big\{T(v,v',m) p (x-v(t-s),v',m,s)
 \vspace{2.5mm}  \\ 
& &  ~\hspace{4cm} \displaystyle  - T(v',v,m) p (x-v(t-s),v,m,s)\big\} dv' dm ds.
    \end{array}
$$
Recall the initial condition \eqref{as:in2} and assumption \eqref{as:T} on $T$,  take the $L^{\infty}(\mathbb{R}^d_x\times V)$ norm on both sides of the above equation to get
$$%\beq\label{pbar}
 \|  \bar{p} (\cdot,\cdot,t) \|_{L^{\infty} (\mathbb{R}^d_x \times V  )} \leq 
 \|\bar p_{0}\|_{L^{\infty}} + 2 V_d C_{\mathcal{T}} \int_{0}^{t} \|  \bar{p} (\cdot,\cdot,s) \|_{L^{\infty} (\mathbb{R}^d_x \times V  )}  ds.
$$%\eeq 
 %where $V_d$ is the volume of $V$. 
 Then, apply Gronwall's inequality, we readily find  \eqref{prop:p2}. 
 
 Next, from the definition of $n$, we have that
$$
\| n (\cdot, t) \|_{L^{\infty}(\mathbb{R}^d_x)}
\leq   V_d   \| \bar{p} (\cdot,\cdot, t) \|_{L^{\infty}(\mathbb{R}^d_x\times V)} ,
$$
thus we have \eqref{nlinf}.

   \begin{itemize}
\item  $L^1$ and $L^\infty$ bounds on $S$.
\end{itemize}

  Recall that  $  S (x,t) = G*n (x,t) $,   by Young's inequality for convolution, it is direct that
\beq\label{sl1}
\| S (\cdot, t) \|_{L^{1}(\mathbb{R}^d_x)} \leq  \|G\|_{L^{1}(\mathbb{R}^d_x)} \| n (\cdot, t) \|_{L^{1}(\mathbb{R}^d_x)} 
= \|p_{0}\|_{L^{1}},
\eeq
where \eqref{nl1} is used,  and
$$
\| S (\cdot, t) \|_{L^{\infty}(\mathbb{R}^d_x)} \leq \|G\|_{L^{1}(\mathbb{R}^d_x)} \| n (\cdot, t) \|_{L^{\infty}(\mathbb{R}^d_x)}
%\leq C \| \bar{p} (\cdot,\cdot, t) \|_{L^{\infty}(\mathbb{R}^d_x\times V)}  V_d 
\leq V_d \|\bar p_{0}\|_{L^{\infty}} (1+2V_dC_{\mathcal{T}}t e^{2V_dC_{\mathcal{T}}t}), 
$$
in which \eqref{nlinf} is used. Then this inequality and \eqref{sl1} gives the estimate \eqref{prop:ps}.   

  \begin{itemize}
\item  $L^\infty$ bound on $  p$.
\end{itemize}

Now we introduce the characteristics of Equation \eqref{frl1}  by
\beq\label{mar11-3}
{d\mathbf{X}\over ds} = v,\quad 
{d\mathbf{V}\over ds }= 0,\quad 
{d\mathbf{M}\over ds} =  F(\mathbf{M}(s),S (\mathbf{X}(s),s)),
\eeq
and along time-backward characteristics starting at $(x,v,m,t)$, we have for $0\leq s\leq t$,
$$
\mathbf{X}(s) = x-v(t-s), \quad \mathbf{M}(s) = m - \int_{s}^{t}F(\mathbf{M}(\tau),S (\mathbf{X}(\tau),\tau)) d\tau,
$$
 integrate \eqref{frl1} along the characteristic from $0$ to $t$ we get
$$
p (x,v,m,t) = p_{0}(\mathbf{X}(0), v, \mathbf{M}(0)) -  
\int_{0}^{t}  \p_{m}   F(\mathbf{M}(\tau),S (\mathbf{X}(\tau),\tau)) 
p (\mathbf{X}(\tau), v, \mathbf{M}(\tau), \tau)  d\tau
$$
\beq\label{eq:p}
+\int_{0}^{t}\int_{V} T(v,v', \mathbf{M}(\tau)) p (\mathbf{X}(\tau), v', \mathbf{M}(\tau), \tau) 
- T(v',v, \mathbf{M}(\tau))   p (\mathbf{X}(\tau), v, \mathbf{M}(\tau), \tau)  dv' d\tau .
\eeq
Now,  take the $L^\infty( \mathbb{R}^d_x \times V \times \mathbb{R}^{+}_m )$ norm on both sides, and under the assumptions \eqref{frl8}-\eqref{as:T}, we have
$$
\|  p (\cdot,\cdot,\cdot,t) \|_{L^{\infty} (\mathbb{R}^d_x \times V \times \mathbb{R}^{+}_m)} \leq \|  p_{0} \|_{L^{\infty} (\mathbb{R}^d_x \times V \times \mathbb{R}^{+}_m)}
$$
\beq\label{j14-1}
+ \big[   \Pi(\max\limits_{0\leq s\leq t}| S(\mathbf{X}(s),s )|)  +  2V_d C_{\mathcal{T}} \big]
\int_{0}^{t}   \|  p (\cdot,\cdot,\cdot,\tau) \|_{L^{\infty} (\mathbb{R}^d_x \times V \times \mathbb{R}^{+}_m)} d\tau,
\eeq
applying Gronwall's inequality again, we obtain  \eqref{prop:p1}.

\subsection{Proof of Theorem \ref{global} }

 We use the Schauder fixed point theorem to  prove   global existence.

\begin{theorem}\label{schauder} (The Schauder fixed point theorem \cite{GT,JS})
Let $X$ be a normed vector space, and let $K\subset X$ be a non-empty, bounded, and convex set. Then for any given continuous mapping $\phi: K \rightarrow K$, with $\phi(K)$ being pre-compact,   there exists a fixed point $x\in K$ such that $\phi(x)=x$.

\end{theorem}

To apply this theorem, we separate the presentation into three steps.\\

\noindent{\bf Step 1.  Setting of the function space and the mapping.}
Fix   $T_{0}>0$, we define 
 $$
  X:=   L^{1} ([0,T_{0}] \times \mathbb{R}^{d}_{x} ),
 $$
and take a bounded convex subset $K$ in $X$ defined by
$$%\begin{multline*}
K= \big\{S\in X \big|S\geq 0, \|S(\cdot, t)\|_{L^{1}} \leq  \| p_{0}\|_{L^{1}}, \
\|S(\cdot, t)\|_{L^{\infty}} \leq       V_d \|\bar p_{0}\|_{L^{\infty}}  (1+2V_d C_{\mathcal{T}}T_0 e^{2V_d C_{\mathcal{T}}T_0}),
   \   \forall t\in [0,T_0] \big\} . 
$$% \end{multline*}
Then we define the mapping $\phi$. We  start from a function $S \in K $, and construct  $p \in  L^{\infty} ([0,T_{0}], L^{1} (\mathbb{R}^{d}_{x}\times V \times R^{+}_{m}) )$ according to \eqref{frl1} with $S$ fixed. This is possible because it is a linear operator and the characteristics are well defined according to \eqref{mar11-3}.

Next, the integration of $p$ with respect to $(v,m)$ defines $n\in  L^{\infty} ([0,T_{0}], L^{1}   (\mathbb{R}^{d}_{x}) )$. By the above a priori estimate \eqref{nl1} and \eqref{nlinf}, we have
 $$
  \|n(\cdot, t)\|_{L^{1}  (\mathbb{R}^{d}_{x}) }
  = \| p_{0}\|_{L^{1}}, \quad\quad\quad \forall t\in [0,T_0],
 $$
 $$
  \|n(\cdot, t)\|_{L^{\infty}  (\mathbb{R}^{d}_{x}) }
  \leq  V_d \|\bar p_{0}\|_{L^{\infty}}  (1+2V_dC_{\mathcal{T}}t e^{2V_dC_{\mathcal{T}}t}),
   \quad \forall t\in [0,T_0].
 $$
Note that by interpolation, for any $q\in (1, \infty)$, 
 \beq\label{sch1}
  \|n(\cdot, t)\|_{L^{q}  (\mathbb{R}^{d}_{x}) }
  \leq   \|n(\cdot, t)\|_{L^{1} (\mathbb{R}^{d}_{x}) }^{1\over q}  
     \|n(\cdot, t)\|_{L^{\infty} (\mathbb{R}^{d}_{x}) }^{q-1\over q} .
 \eeq
  Further, we  construct  $\Sigma$ as the solution to  
  \beq\label{sch7}
   - \Delta \Sigma + \Sigma = n:= \int_{0}^{\infty}\int_{V} p (x,v,m,t) dv dm.
   \eeq
By using Besssel potential and Young's inequality for convolution (as the proof of \eqref{sl1}),  we have
$$
 \|\Sigma(\cdot, t)\|_{L^{1}} \leq  \| p_{0}\|_{L^{1}},  \quad
\|\Sigma(\cdot, t)\|_{L^{\infty}} \leq  V_d \|\bar p_{0}\|_{L^{\infty}}  (1+2V_dC_{\mathcal{T}}T_0 e^{2V_dC_{\mathcal{T}}T_0}),
   \quad \forall t\in [0,T_0],
$$
thus it holds that $\Sigma \in K$, therefore  we have defined a mapping  on $K$ that
$$ \phi:   S \mapsto \Sigma. $$  

\noindent{\bf Step 2. Continuity}. To prove the continuity of the mapping $\phi: S \mapsto \Sigma$, we observe that, firstly, the mapping 
$$\phi_1: p \mapsto \Sigma$$
 defined by \eqref{sch7} is continuous, because it is a bounded linear operator from $L^q$ to $L^q$ for $1\leq q \leq \infty$.
Secondly, the mapping
$$\phi_2: S \mapsto  p$$ is also continuous, which can be seen exactly from the representation formula \eqref{eq:p} (and also \eqref{j14-1}). Indeed, the characteristics are uniquely defined and continuous with respect to parameters, although they are not Lipschitz, neither relevant for DiPerna-Lions theory \cite{DL}.

 In conclusion, the composition $\phi: S \mapsto \Sigma$ is   continuous.\\

%\vspace{3mm}

\noindent{\bf Step 3. Compactness of the mapping}.
To use the Schauder fixed point theorem,  we  need further to prove that 
   \beq\label{sch9}
   \phi(K) \text{ is pre-compact  in } K   .
 \eeq
  
The proof of \eqref{sch9} consists of the following three claims. 

\vspace{3mm}
 {\bf  Claim 1.} Local compactness in space:   
$$\Sigma(\cdot, t) \in W^{2,q} (\mathbb{R}^{d}_{x}), \quad \forall t\in [0,T_0],\quad \forall 1<q<\infty. $$ 
This claim is clear by using standard elliptic regularity estimate \cite{GT}, since $\Sigma$ is defined by \eqref{sch7} and recall \eqref{sch1} that $n(\cdot, t)\in L^q  (\mathbb{R}^{d}_{x})$, for all $q>1$.

\vspace{3mm}
 {\bf  Claim 2.} Local compactness in time:   
$$\p_{t}\Sigma(\cdot, t) \in W^{1,q} (\mathbb{R}^{d}_{x}), \quad \forall t\in [0,T_0],\quad \forall 1<q<\infty. $$
This is because, take time derivative of the elliptic equation for $\Sigma$ we have
 $$
 - \Delta \p_{t}\Sigma + \p_{t}\Sigma = \p_{t}n = \nabla_{x}\cdot \int_{0}^{\infty} \int_{V} vp dvdm.
 $$
Note that for any $t$, the integration on the right hand side of above is in $ L^{q} (\mathbb{R}^{d}_{x})$, for any $q>1$, the standard elliptic regularity theory shows that $\p_{t}\Sigma \in W^{1,q} (\mathbb{R}^{d}_{x})$,
 and actually 
 $$
  \|\p_{t}\Sigma(\cdot, t)\|_{W^{1,q} (\mathbb{R}^{d}_{x}) } \leq 
 \bar C \| \int_{0}^{\infty} \int_{V} vp \ dvdm \|_{L^{q} (\mathbb{R}^{d}_{x}) }
 \leq  \bar C  \  diam(V)   \    \| \int_{V} \bar p dv \|_{L^{q} (\mathbb{R}^{d}_{x}) } 
 $$
 $$
   \leq  \tilde C  \    \| \int_{V} \bar p dv \|_{L^1\cap L^\infty (\mathbb{R}^{d}_{x}) }
   \leq \tilde C  \    \|\bar p \|_{L^1\cap L^\infty (\mathbb{R}^{d}_{x}\times V) },  \quad  \forall t\in[0,T_{0}],
 $$
 where $\bar C$ independent of $S$ is a constant given by elliptic estimate and $\tilde C$ is a genetic constant which differs from line to line.

\vspace{3mm}
 {\bf  Claim 3.} Control for  $|x|\sim \infty$:  
 \beq\label{sch2}
 \int_{\mathbb{R}^{d}_{x}} \langle x \rangle \Sigma dx  < \infty, \quad \forall t\in[0,T_{0}], 
 \quad \text{where}\quad \langle x \rangle := \sqrt{1+|x|^{2}}.
\eeq
For this, note that $\Sigma$ satisfies \eqref{sch7}, then
  \beq\label{sch5}
 \int_{\mathbb{R}^{d}_{x}} \langle x \rangle \Sigma dx = 
  \int_{\mathbb{R}^{d}_{x}} \langle x \rangle n dx +
   \int_{\mathbb{R}^{d}_{x}} \langle x \rangle \Delta \Sigma dx ,
 \eeq
in which the second term is bounded because
$$
  \int_{\mathbb{R}^{d}_{x}} \langle x \rangle \cdot \Delta \Sigma dx =
    \int_{\mathbb{R}^{d}_{x}} \Delta \langle x \rangle \cdot \Sigma dx =
 \int_{\mathbb{R}^{d}_{x}}  \big({d-1\over \langle x \rangle } + {1\over \langle x \rangle^{3} }\big) 
  \cdot \Sigma dx
\leq d\|\Sigma\|_{L^{1}(\mathbb{R}^{d}_{x})}.
$$
To control the first term on the right hand side of \eqref{sch5}, we multiply Equation \eqref{frl1} by $\langle x \rangle $ and integrate:
$$
 {d\over dt} \int_{\mathbb{R}^d_x} \int_{0}^{\infty}\int_{V} \langle x \rangle p dvdmdx + \int_{\mathbb{R}^d_x} \int_{0}^{\infty}\int_{V} \langle x \rangle v\cdot \nabla_{x}p dvdmdx=0.
$$
Then
$$
{d\over dt} \int_{\mathbb{R}^d_x}   \langle x \rangle n dx 
= {d\over dt} \int_{\mathbb{R}^d_x} \int_{0}^{\infty}\int_{V} \langle x \rangle p dvdmdx
= \int_{\mathbb{R}^d_x} \int_{0}^{\infty}\int_{V} {v\cdot x \over \langle x \rangle } p  dvdmdx
\leq  \| p_{0}\|_{L^{1}} .
$$
Therefore from the initial date \eqref{as:in1} we have the bound
$$
\int_{\mathbb{R}^d_x}   \langle x \rangle n dx \leq \int_{\mathbb{R}^d_x}   \langle x \rangle n_{0} dx
 +t \| p_{0}\|_{L^{1}} =   \|  \langle x \rangle p_{0}\|_{L^{1}} + t \| p_{0}\|_{L^{1}}, \  \forall t\in[0,T_{0}],
$$
then we have proved \eqref{sch2}, and further,
\beq\label{sch6}
\langle x \rangle \Sigma \in X = L^{1} ([0,T_{0}] \times \mathbb{R}^{d}_{x} ).
\eeq
The property \eqref{sch6} yields:
\beq\label{sch8}
\forall \e_{0}>0, \  \exists \Omega \subset [0,T_{0}] \times \mathbb{R}^{d}_{x} , 
\text{~bounded, measurable, such that~}  
\|  \Sigma \| _{  L^{1} ([0,T_{0}] \times \mathbb{R}^{d\ }_{x }  \setminus \Omega)}
<  \e_{0}.
\eeq
This is because
$$
 \int_{0}^{T_{0}} \int_{|x|\geq R} \Sigma dxdt  \leq 
 {1\over \langle R \rangle}  \int_{0}^{T_{0}} \int_{|x|\geq R} \langle x \rangle \Sigma dxdt
 \rightarrow 0 \text{~as~} R \rightarrow \infty.
$$

\vspace{3mm}

In conclusion, for any fixed $\Omega$ with finite measure defined in \eqref{sch8}, by Claim 1 and Claim 2, we have
$$\{\Sigma_{|\Omega}\}  \text{ is precompact in } L^1(\Omega),$$
together with \eqref{sch8}, the weighted in space control away from $\Omega$, 
we proved  \eqref{sch9}, 
by using strong compactness criterion in Brezis \cite{Brezis} (Theorem 4.26, Corollary 4.27, p. 111). \\

Now we use the Schauder Theorem  \ref{schauder} to get a fixed point which proves existence of solution  $S\in K\subset X=  L^{1} ([0,T_{0}] \times \mathbb{R}^{d}_{x} )$ for all $T_{0} >0$. And next, we can also construct the global solution $p\in  L^{\infty} ([0,T_{0}], L^{1}\cap L^{\infty}(\mathbb{R}^{d}_{x}\times V \times \mathbb{R}_{m}^{+} ))$ for all $T_{0} >0$ by the above a priori estimates.  The proof of Theorem \ref{global} is complete.
  $\hfill\square$

%\begin{remark}
%By imposing higher regularities on initial data and $F$ and $T$, we can similarly improve the regularity of the global solution and show that the weak is in fact a classical solution. 
%\end{remark}

%\newpage

\section{Regularity and uniqueness of the solution}\label{reguni}

In above section, the existence of global weak solution is derived by using the Schauder fixed point theorem.  
However, the Schauder fixed point theorem does not provide uniqueness. Therefore, additional analysis is required to be developed to obtain uniqueness. 

Note the property $(P4)$ of $G$ in Proposition \ref{BP} and the a priori estimate \eqref{nlinf},   by using Young's inequality we have 
\beq\label{apr6-4}
\|\p_{x_i} S \|_{L^{\infty} (\mathbb{R}^d_x  )}   \leq  
\| \p_{x_i} G\|_{L^{1} (\mathbb{R}^d_x  )}   \| n\|_{L^{\infty} (\mathbb{R}^d_x  )}
< \infty.
\eeq
Also recall \eqref{prop:ps} that $S (\cdot,t) \in  L^{\infty} (\mathbb{R}^d_x)$, then we obtain  
 \begin{proposition}\label{apr18-1}
 Let $S$ be the weak solution of   \eqref{frl1}-\eqref{frl3}   defined in  Theorem \ref{global}. Then
$$
S (\cdot,t) \in  W^{1,\infty} (\mathbb{R}^d_x)     .%        , \quad 1< q < \infty.
$$
\end{proposition}

By further assumptions on regularity of coefficients and initial data, we can get higher regularity of $p$. For this purpose, we further assume
\begin{itemize}
\item   There  exists a non-negative, non-decreasing continuous function $\tilde \Pi(\cdot)\in C(\mathbb{R})$ %and a positive constant $C_{\mathcal{F}} $ 
such that
\beq\label{apr6-1}
%| F(m,S)| +
  |\p_{S}F(m,S)| + |\p_{mm}F(m,S)| +  |\p_m\p_{S}F(m,S)| \leq   \tilde\Pi(S).       %     C_{\mathcal{F}}(1+ \Pi(S)).
\eeq
\item There exist a positive constant $\tilde C_{\mathcal{T}}$ such that 
\beq\label{apr6-2}
 |\nabla_v T(v,v',m)|  +  |\nabla_{v'} T(v,v',m)|  +  |\p_m T(v,v',m)|\leq \tilde C_{\mathcal{T}}. % (1+ \Lambda(m)),
\eeq
%where $\Lambda \in C(\mathbb{R^+})$ is a non-negative, non-decreasing continuous function. 
\end{itemize}

\begin{proposition}\label{apr7-1}

Under the assumptions on Theorem \ref{global}. Further  assume \eqref{apr6-1}-\eqref{apr6-2} and let initial data 
$p_0 \in W^{1,1} % (\mathbb{R}^d_x \times V \times \mathbb{R}^{+}_m)   
\cap W^{1,\infty} (\mathbb{R}^d_x \times V \times \mathbb{R}^{+}_m)   $. Then the weak solution of  \eqref{frl1}-\eqref{frl3}  satisfies
\beq\label{apr6-3}
p(\cdot,\cdot,\cdot,t) \in  W^{1,1} % (\mathbb{R}^d_x \times V \times \mathbb{R}^{+}_m)   
\cap W^{1,\infty} (\mathbb{R}^d_x \times V \times \mathbb{R}^{+}_m) .
\eeq

\end{proposition}

\begin{proof}  {\it 1. Regularity on $x$.} 
  Differentiate  \eqref{frl1} with respect to $x_i$, $1\leq i \leq d$, we have
$$
 \p_t \p_{x_i} p   +  v \cdot \nabla_x \p_{x_i} p   +   F(m, S )  \p_m  \p_{x_i} p   =
 - 2  \p_m F(m, S )   \p_{x_i} p  -  \p_m \p_S F(m, S ) \  \p_{x_i} S \ p 
 $$
 $$
+  \int_V T(v,v',m) \p_{x_i} p  (x,v',m,t) - T(v',v,m) \p_{x_i} p  (x,v,m,t) dv'.
$$
Integrate along the characteristics defined by \eqref{mar11-3} from $0$ to $t$:  
$$
 \p_{x_i} p  (x,v,m,t) = \p_{x_i} p_0 (\mathbf{X}(0), v, \mathbf{M}(0)) -  
\int_{0}^{t}  2 \p_{m}   F(\mathbf{M}(\tau),S (\mathbf{X}(\tau),\tau)) 
 \p_{x_i} p  (\mathbf{X}(\tau), v, \mathbf{M}(\tau), \tau)  d\tau
$$
$$%\beq\label{apr6-5}
 -  \int_{0}^{t}   \ \p_ {m}    \p_S F(\mathbf{M}(\tau),S (\mathbf{X}(\tau),\tau))  \p_{x_i} S (\mathbf{X}(\tau),\tau)
p  (\mathbf{X}(\tau), v, \mathbf{M}(\tau), \tau)  d\tau
$$%\eeq
$$
+\int_{0}^{t}\int_{V} T(v,v', \mathbf{M}(\tau)) \p_{x_i} p  (\mathbf{X}(\tau), v', \mathbf{M}(\tau), \tau) 
- T(v',v, \mathbf{M}(\tau))   \p_{x_i} p  (\mathbf{X}(\tau), v, \mathbf{M}(\tau), \tau)  dv' d\tau .
$$
 Using the assumptions on the coefficients and initial data as stated above, we get  
 $$
 \begin{array}{rl}
| \p_{x_i} p  (x,v,m,t)| \leq & \displaystyle | \p_{x_i} p_0(\mathbf{X}(0), v, \mathbf{M}(0))| +  
 2 \Pi(\max\limits_{0\leq s\leq t}| S(\mathbf{X}(s),s )|)  \int_{0}^{t}  
 | \p_{x_i} p  (\mathbf{X}(\tau), v, \mathbf{M}(\tau), \tau) | d\tau
\vspace{2.5mm}  \\ 
& \displaystyle +  \tilde \Pi(\max\limits_{0\leq s\leq t}| S(\mathbf{X}(s),s )|)         \|\p_{x_i} S \|_{L^{\infty} (\mathbb{R}^d_x  )}  \int_{0}^{t}    
|p  (\mathbf{X}(\tau), v, \mathbf{M}(\tau), \tau)  |d\tau
\vspace{2.5mm}  \\ 
&\displaystyle  +    2V_d C_{\mathcal{T}}     \int_{0}^{t}   | \p_{x_i} p  (\mathbf{X}(\tau), v', \mathbf{M}(\tau), \tau) |
  d\tau ,
\end{array}
$$
that is,
\beq\label{apr6-6}
\begin{array}{rcl}
 |\p_{x_i} p  (x,v,m,t)| &\leq&  |\p_{x_i} p_0(\mathbf{X}(0), v, \mathbf{M}(0))| 
 \vspace{2.5mm}  \\ 
 & &
 +\displaystyle   2 \big[    \Pi(\max\limits_{0\leq s\leq t}| S(\mathbf{X}(s),s )|)      + V_d C_{\mathcal{T}} \big]
    \int_{0}^{t}     |\p_{x_i} p   (\mathbf{X}(\tau), v, \mathbf{M}(\tau), \tau)  |d\tau
\vspace{2.5mm}  \\ 
 & &
 +\displaystyle  \tilde  \Pi(\max\limits_{0\leq s\leq t}| S(\mathbf{X}(s),s )|)        \|\p_{x_i} S \|_{L^{\infty} (\mathbb{R}^d_x  )}   
    \int_{0}^{t}    |p  (\mathbf{X}(\tau), v, \mathbf{M}(\tau), \tau)  |d\tau.
    \end{array}
\eeq

{\it 2. Regularity on $v$.}     Differentiate  \eqref{frl1} with respect to $v_i$,  $1\leq i \leq d$, 
we have
$$
 \p_t \p_{v_i} p   +  v \cdot \nabla_x \p_{v_i} p   +   F(m, S )  \p_m  \p_{v_i} p   =
 -  \p_m F(m, S )   \p_{v_i} p    - \p_{x_i}  p
 $$
 $$
 \hspace{2cm}  +  \int_V T(v,v',m) \p_{v_i} p  (x,v',m,t) - T(v',v,m) \p_{v_i} p  (x,v,m,t) dv'
$$
 $$
 \hspace{2cm}  +  \int_V  \p_{v_i} T (v,v',m) p (x,v',m,t) -  \p_{v_i} T (v',v,m) p  (x,v,m,t) dv'.
$$
Similarly,  integrate along the characteristics defined by \eqref{mar11-3} from $0$ to $t$, and use the assumptions, we obtain
 \beq\label{apr6-7}
\begin{array}{rl}
 |\p_{v_i} p  (x,v,m,t)| \leq& \displaystyle  |\p_{v_i} p_0(\mathbf{X}(0), v, \mathbf{M}(0))|  
   +    \int_{0}^{t}     |\p_{x_i} p   (\mathbf{X}(\tau), v, \mathbf{M}(\tau), \tau)  |d\tau
 \vspace{2.5mm}  \\ 
 & \displaystyle + \Pi(\max\limits_{0\leq s\leq t}| S(\mathbf{X}(s),s )|)  
  \int_{0}^{t}     |\p_{v_i} p   (\mathbf{X}(\tau), v, \mathbf{M}(\tau), \tau)  |d\tau
 \vspace{2.5mm}  \\ 
  &
 +\displaystyle    2 V_d  \tilde C_{\mathcal{T}} 
    \int_{0}^{t}    |p  (\mathbf{X}(\tau), v, \mathbf{M}(\tau), \tau)  |d\tau
   +   V_d   C_{\mathcal{T}}   
    \int_{0}^{t}     |\p_{v_i} p   (\mathbf{X}(\tau), v, \mathbf{M}(\tau), \tau)  |d\tau  .
    \end{array}
\eeq 

{\it 3. Regularity on $m$.}  Next,  differentiate  \eqref{frl1} with respect to $m$, 
we have
$$
 \p_t  \p_ {m} p   +  v \cdot \nabla_x  \p_ {m} p   +   F(m, S )  \p_m   \p_ {m} p   =
 -  \p_{mm}  F(m, S )   p   -  2 \p_ {m}  F(m, S )  \p_ {m} p
 $$
 $$
 \hspace{2cm}  +  \int_V T(v,v',m)  \p_ {m} p  (x,v',m,t) - T(v',v,m)  \p_ {m} p  (x,v,m,t) dv'
$$
 $$
 \hspace{2cm}  +  \int_V T_{m}(v,v',m) p (x,v',m,t) - T_{m}(v',v,m) p  (x,v,m,t) dv'.
$$
Integrate along the characteristics defined by \eqref{mar11-3} from $0$ to $t$, we obtain
 \beq\label{apr6-8}
\begin{array}{rcl}
| \p_ {m} p  (x,v,m,t)| &\leq&  |\p_ {m} p_{0}(\mathbf{X}(0), v, \mathbf{M}(0))| 
 \vspace{2.5mm}  \\ 
 & &
 +\displaystyle    \big[   2 \Pi(\max\limits_{0\leq s\leq t}| S(\mathbf{X}(s),s )|)      + 2 V_d C_{\mathcal{T}} \big]
    \int_{0}^{t}    | \p_ {m} p   (\mathbf{X}(\tau), v, \mathbf{M}(\tau), \tau)  |d\tau
\vspace{2.5mm}  \\ 
 & &
 +\displaystyle  \big[   \tilde  \Pi(\max\limits_{0\leq s\leq t}| S(\mathbf{X}(s),s )|)  + 2   V_d \tilde C_{\mathcal{T}} \big]     
    \int_{0}^{t}    |p  (\mathbf{X}(\tau), v, \mathbf{M}(\tau), \tau)  |d\tau.
    \end{array}
\eeq

In summary, recall the estimate  \eqref{prop:p1} for $p$ and Proposition \ref{apr18-1} for $S$,  combining  \eqref{apr6-6}-\eqref{apr6-8} and taking $L^q$ norms,      using  Gronwall's inequality, we deduce, for $1\leq q \leq \infty$, $\forall \ t>0$,
$$
\|\nabla _x p (\cdot,\cdot,\cdot,t) \|_{L^q} +  \|\nabla _v  p (\cdot,\cdot,\cdot,t) \|_{L^q} + \| \p_ {m} p (\cdot,\cdot,\cdot,t) \|_{L^q}
 \leq C( \|S(\cdot, t) \|_{W^{1,\infty}},    \|p_0 \|_{W^{1,q}}, V_d, C_{\mathcal{T}}, \tilde C_{\mathcal{T}})
 < \infty. 
$$
Together with  \eqref{prop:p1}, we have \eqref{apr6-3}.  $\hfill\square$\\
\end{proof}

For later use, we derive the following a priori estimates. 
 
 \begin{proposition}\label{apr18-2}

Under the assumptions on Proposition \ref{apr7-1}, and further assume 
 the initial data satisfy 
$\p_m p_0 \in L^\infty_x ( L^1_{v,m} ) $. Then for all $t>0$, 
 $$
 \|p\|_{L^\infty_x ( L^1_{v,m} )} +  \|\p_m p \|_{L^\infty_x ( L^1_{v,m} )}  < \infty. 
 $$

\end{proposition}

\begin{proof} Firstly, recall the a priori estimate \eqref{prop:p2} on $\bar p$ thus $p\in L^\infty_{x,v} ( L^1_m) $.
Note $V$ is compact, then we readily have
$$
 \|p\|_{L^\infty_x ( L^1_{v,m}) }   < \infty. 
 $$
Next we prove $p_m  \in L^\infty_x ( L^1_{v,m}  )$. For this, we  take $L^\infty_x  (L^1_{v,m} ) $ norm on both sides of   \eqref{apr6-8}, 
then the result follows by using again  Gronwall's inequality.  
  $\hfill\square$\\
\end{proof}

 We now state the uniqueness result.  \begin{proposition}\label{apr7-2}

Under the assumptions on Proposition \ref{apr18-2}, the weak solution of   \eqref{frl1}-\eqref{frl3} is unique.

\end{proposition}

\begin{proof}
We follow the steps in \cite{LP}. 
Assume both $(p_1,S_1)$ and $(p_2,S_2)$ be  weak solutions of  \eqref{frl1}-\eqref{frl3} with same initial data $p_0$ satisfying  \eqref{as:in1}-\eqref{as:in2}. Denote 
$$\tilde p = p_1 - p_2, \quad \tilde   S= S_1 - S_2. $$
Then we have
\beq \label{apr7-3}
\begin{array}{c}
\displaystyle  \p_t \tilde  p  +  v \cdot \nabla_x \tilde  p  + \p_m [  F(m, S_1 )  \tilde   p  + ( F(m, S_1 ) -  F(m, S_2 ))  p_2 ] 
 \vspace{2.5mm}  \\ 
\hspace{4cm} \displaystyle  = \int_V T(v,v',m) \tilde  p (x,v',m,t) - T(v',v,m) \tilde  p (x,v,m,t) dv',
 \end{array}
\eeq
\beq \label{apr7-4}
-\Delta \tilde  S  + \tilde  S  = \tilde  n (x,t):= \int_{0}^{\infty}\int_{V} \tilde  p (x,v,m,t) dv dm,
\eeq
and the initial data 
\beq \label{apr7-5}
\tilde p = 0.
\eeq
From \eqref{apr7-4} we write $  \tilde S  = G * \tilde  n $. By Young's inequality,    
\beq \label{apr18-3}
\| \tilde S \|_{L^1_x} \leq          \| \tilde  p(t) \|_{L^1_{x,v,m}} .
\eeq
Next, from \eqref{apr7-3} we have
$$
\begin{array}{rl}
 &\displaystyle  \p_t \tilde  p  +  v \cdot \nabla_x \tilde  p  +  F(m, S_1 )   \p_m \tilde   p    
 \vspace{2.5mm}  \\ 
   =& \displaystyle - \p_m   F(m, S_1 )  \tilde   p   - \p_m \Big(   p_2   \int_0^1   \p_S F(m, (1-\theta)S_2  + \theta S_1) d \theta  \Big)   \tilde S  
  \vspace{2.5mm}  \\ 
  & \displaystyle  + \int_V T(v,v',m) \tilde  p (x,v',m,t) - T(v',v,m) \tilde  p (x,v,m,t) dv'.
 \end{array}
$$
Integrate along the characteristics defined by \eqref{mar11-3} from $0$ to $t$,  with $S$ replaced by $S_1$, we get
$$
\tilde p  (x,v,m,t) = \tilde  p_{0}(\mathbf{X}(0), v, \mathbf{M}(0))    
 -  \int_{0}^{t}    \p_{m}   F(\mathbf{M}(\tau),S_1 (\mathbf{X}(\tau),\tau)) 
\tilde p (\mathbf{X}(\tau), v, \mathbf{M}(\tau), \tau)  d\tau
$$
$$%\beq\label{apr6-5}
 -  \int_{0}^{t}  \p_m \Big(   p_2   \int_0^1   \p_S F(m, (1-\theta)S_2  + \theta S_1) d \theta  \Big)  
   (\mathbf{X}(\tau), v, \mathbf{M}(\tau), \tau)           \tilde S(\mathbf{X}(\tau),\tau)  d\tau
$$%\eeq
$$
+\int_{0}^{t}\int_{V} T(v,v', \mathbf{M}(\tau)) \tilde p  (\mathbf{X}(\tau), v', \mathbf{M}(\tau), \tau) 
- T(v',v, \mathbf{M}(\tau))  \tilde  p  (\mathbf{X}(\tau), v, \mathbf{M}(\tau), \tau)  dv' d\tau .
$$
 Take $L^1_{x,v,m}$ norm  on both sides, we obtain
 $$
 \begin{array}{rl}
\| \tilde  p(t) \|_{L^1_{x,v,m}} \leq   &  \| \tilde  p_0 \|_{L^1_{x,v,m}}  
+ \displaystyle  \big[   \Pi(\max\limits_{0\leq s\leq t}| S(\mathbf{X}(s),s )|)  +  2V_d C_{\mathcal{T}} \big]  
\int_0^t   \| \tilde  p(s) \|_{L^1_{x,v,m}} ds 
\vspace{2.5mm}  \\ 
& +  \Big( \|p\|_{L^\infty_x  (L^1_{v,m}) } +  \|p_m\|_{L^\infty_x  (L^1_{v,m}) }\Big)  \Pi(\max\limits_{0\leq s\leq t}| S(\mathbf{X}(s),s )|) \ 
\| \tilde S \|_{L^1_x}.  
\end{array}
 $$ 
Notice from Proposition \ref{apr18-2} that   $p, p_m  \in L^\infty_x ( L^1_{v,m} ) $, and use \eqref{apr18-3}, we get
$$
\| \tilde  p(t) \|_{L^1_{x,v,m}} \leq      \| \tilde  p_0 \|_{L^1_{x,v,m}}   + C(t) \int_0^t   \| \tilde  p(s) \|_{L^1_{x,v,m}} ds ,
$$
where $C(t)$ is bounded for any given $t$. 
Recall the zero initial data \eqref{apr7-5}, thus by applying Gronwall's inequality we have
$$
\| \tilde  p(t) \|_{L^1_{x,v,m}}  \equiv 0,
$$
 then we have proved the uniqueness of the global weak solution.
 $\hfill\square$
\end{proof}

% \newpage

%%%%%%%%%%%%%%%%
%----------------------------------------
\section{Fast adaptation limit}\label{section3}
%---------------------------------------
%%%%%%%%%%%%%%%%

 In this section we investigate, as $\e\rightarrow0$, the limiting behavior of the system
 \beq \label{newepsilon}
 \left\{
 \begin{array}{rcl}
  \p_t p^{\varepsilon} +  v \cdot \nabla_x p^{\varepsilon} + \p_m [  \f{F(m, S^{\varepsilon})}{\e}  p^{\varepsilon}]& = &{\dis \int}_V T(v,v',m) p^{\varepsilon}(x,v',m,t) - T(v',v,m) p^{\varepsilon}(x,v,m,t) dv'
\vspace{2mm}  \\
 -\Delta S^{\varepsilon} +S^{\varepsilon} &=& n^{\varepsilon}(x,t):= {\dis \int_{0}^{\infty}  \int_{V}} p^{\varepsilon}(x,v,m,t) dv dm, \vspace{2mm} \\
 p^\e (x,v, m=0,t)=0, & & p^\e (x,v, m=+\infty,t)=0,
\end{array}
\right.
\eeq
with initial data $p_0$ which satisfies \eqref{as:in1}-\eqref{as:in2}.
Below we denote  $\mathcal{M}(\Omega)$ the space of Radon measures on $\Omega$, $C_0$ is the Banach space of continuous functions which vanishe at $\infty$,  and  the notation
$$
\bar{p}^\e (x,v,t) := \int_{0}^{\infty} p^\e  (x,v,m,t)  dm .
$$

First of all, we recall the a priori estimates including the parameter $\e$, which is 

\begin{proposition}

Under the same assumptions  in Theorem \ref{global}.  The  solution to \eqref{newepsilon} satisfies
   \beq\label{newprop:p1}
\|  p^\e   (\cdot,\cdot,\cdot,t) \|_{ L^{\infty} (\mathbb{R}^d_x \times V \times \mathbb{R}^{+}_m)} \leq  
\|p_{0}\|_{  L^{\infty}}  \big( 1+ {\tilde{C}t \over \e}   e^{\tilde{C} t\over \e}  \big), \quad \forall t>0,
\eeq
and
  \beq\label{mar12}
\|  p^\e   (\cdot,\cdot,\cdot,t) \|_{ L^{1} (\mathbb{R}^d_x \times V \times \mathbb{R}^{+}_m)} = \|p_{0}\|_{L^{1}} , \quad \forall t>0,
\eeq
\beq\label{newprop:ps}
\|  S ^\e  (\cdot,t) \|_{L^{1}\cap L^{\infty} (\mathbb{R}^d_x)} \leq  \|p_{0}\|_{L^{1}} +
  V_d  \|\bar p_{0}\|_{L^{\infty}} (1+2V_dC_{\mathcal{T}}t e^{2V_dC_{\mathcal{T}}t}), \quad \forall t>0,
\eeq
\beq\label{newnlinf}
\| n^\e   (\cdot,t) \|_{L^{\infty} (\mathbb{R}^d_x  )} \leq V_d \|\bar p_{0}\|_{L^{\infty}}  (1+2V_dC_{\mathcal{T}}t e^{2V_dC_{\mathcal{T}}t}),\quad \forall t>0,
\eeq
\beq\label{newprop:p2}
\|  \bar{p}^\e   (\cdot,\cdot,t) \|_{L^{\infty} (\mathbb{R}^d_x \times V  )} \leq  \|\bar p_{0}\|_{L^{\infty}}  (1+2V_dC_{\mathcal{T}}t e^{2V_dC_{\mathcal{T}}t}),\quad \forall t>0.
\eeq

\end{proposition}
From the above proposition, we have
\begin{lemma}\label{l3.1} 
Let $(p^\e, S^\e)$ be solution of \eqref{newepsilon} and   further assume \eqref{as:F1}-\eqref{mar11-4}. Fix any $T>0$. 
%there exist  $p(\cdot,\cdot,\cdot, t)\in \mathcal{M}( \mathbb{R}^d_x \times V \times \mathbb{R}^{+}_m )$,  $\bar p(\cdot,\cdot, t)\in \mathcal{M}( \mathbb{R}^d_x \times V)$ amd  $S(\cdot,t) \in \mathcal{M}( \mathbb{R}^d_x ) $, such that,
Then, up to a subsequence, we have
\beq\label{j20-1}
% p^{\e}(x, v, m, t)  \rightharpoonup  p(x, v, m, t)  \quad\text{in}\quad L^\infty ( [0,T],  \mathcal{M}( \mathbb{R}^d_x \times V \times \mathbb{R}^{+}_m  ) -w* ),
p^{\e}    \rightharpoonup p  \quad\text{in}\quad  \mathcal{M}( \mathbb{R}^d_x \times V \times \mathbb{R}^{+}_m \times  [0,T]) -w*,
\eeq
\beq\label{j20-2}
%\bar p^{\e}    \rightharpoonup \bar p  \quad\text{in}\quad  \mathcal{M}( \mathbb{R}^d_x \times V \times[0,T] ) -w* ,
\bar p^{\e}     \rightharpoonup \bar p    \quad\text{in}\quad L^\infty ( \mathbb{R}^d_x \times V \times  [0,T] ) -w* ,
\eeq
\beq\label{j20-5}
 \bar p(x,v, t)=  \int_{0}^{\infty} p  (x,v,m,t)  dm ,
 % \quad\text{in the sense of distribution, }
\eeq
\beq\label{j20-3}
S^\e (x, t)   \rightarrow   S(x, t) \quad\text{strong in}\quad  L^q( \mathbb{R}^d_x \times  [0,T] ) ,\quad \forall 1\leq q<\infty  ,
\eeq
%{\color{red}
and further 
\beq\label{f3-1}
S^\e (x, t)   \rightarrow   S(x, t) \quad\text{ in}\quad  C_0( \mathbb{R}^d_x \times  [0,T] )    .
\eeq
%}

\end{lemma}
We postpone the proof of Lemma \ref{l3.1} to the end of this section. The main theorem in this section is
\begin{theorem}\label{f13-1}
Under the assumptions in Lemma \ref{l3.1}. Let $(p, \ \bar p,\  S)$ be defined in Lemma \ref{l3.1}. Then
\beq\label{frl14}
 p(x,v,m,t)=\bar{p}(x,v,t) \delta(m-m_0({S} )),
\eeq
where $m_0(S)$ is defined in \eqref{as:F1} such that $F(m_0(S),{S})=0$.  
And $(\bar p,  S)$ satisfies, in the sense of distribution, 
\beq\label{frl13}
\p_t \bar{p} +   v\cdot\nabla_x\bar{p} 
=  \int_V T(v,v',m_0(S))  \bar{p}(x,v',t) -  T(v',v,m_0(S))   \bar{p}(x,v,t) dv' , 
\eeq
\beq\label{j20-4}
 -\Delta S  +S = n (x,t):=  {\dis \int_{V}} \bar p (x,v,t) dv  .
 \eeq
 The initial data of the resulting limit equation \eqref{frl13} is given by $\bar p_0 = \int_0^\infty p_0(x,v,m) dm$.

\end{theorem}

\begin{proof}
%{\color{red}
Multiply Equation \eqref{newepsilon}$_1$ by $\varepsilon$ to get
$$%\beq\label{frl9}
 \p_m [  F(m, S^{\varepsilon})   p^{\varepsilon}] 
= \e \int_V [ T(v,v',m) p^{\varepsilon}(x,v',m,t) - T(v',v,m)  p^{\varepsilon}(x,v,m,t)] dv' 
-\varepsilon \p_t p^{\varepsilon} - \varepsilon v \cdot \nabla_x p^{\varepsilon} .
$$%\eeq
Recall we had \eqref{j20-1} that $p^{\e}    \rightharpoonup p$ weak* in   $ \mathcal{M}( \mathbb{R}^d_x \times V \times \mathbb{R}^{+}_m \times  [0,T]) $, 
 and \eqref{f3-1} that $S^\e (x, t)   \rightarrow   S(x, t) $    in  $C_0( \mathbb{R}^d_x \times  [0,T] ) $, 
 therefore we can pass to the limit in weak sense as $\e \rightarrow 0$, up to a subsequence, to get             %  }
\beq\label{frl10}
 \p_{m}   [ F(m, S)   p] = 0  \quad\text{in the sense of measure, }
\eeq
thus
$$
F(m, S)   p  \equiv constant \equiv 0  \quad\text{in the sense of measure. }
$$
Using  the assumption \eqref{as:F1} on $F$, also using the fact \eqref{j20-5},  we  readily  conclude \eqref{frl14}.

On the other hand, take the integration of \eqref{newepsilon}$_1$ with respect to $m$ we have
$$
\p_{t} \bar{p}^{\e} + v\cdot \nabla _{x}\bar{p}^{\e} =  \int_{0}^{\infty} \int_V T(v,v',m) p^{\varepsilon}(x,v',m,t) -  T(v',v,m) p^{\varepsilon}(x,v,m,t)] dv'dm.
 $$
By  \eqref{j20-2}, we pass to the limit as $\e\to 0$ up to a subsequence,  and find, in the sense of distribution,
 $$
\p_{t} \bar{p} + v\cdot \nabla _{x}\bar{p}  =  \int_{0}^{\infty} \int_V T(v,v',m) p (x,v',m,t) - T(v',v,m) p (x,v,m,t)] dv' dm.
 $$
 Recall that $p$ has the form in \eqref{frl14}, then the equation of $\bar{p}$ is thus \eqref{frl13}. 
 
 Next, \eqref{j20-4} is derived by passing to the limit as $\e\to 0$ up to a subsequence on \eqref{newepsilon}$_2$, using the property \eqref{j20-3}.
      $\hfill\square$

\end{proof}

\vspace{3mm}

\begin{proof}   {\it Proof of Lemma \ref{l3.1}.}

{\bf 1. }Recall the a priori estimate,  $\{p^{\e} \}$ is bounded in $  L^{1}( \mathbb{R}^d_x \times V \times \mathbb{R}^{+}_m \times[0,T] ) $, thus $w*$ precompact in  $\mathcal{M}( \mathbb{R}^d_x \times V\times \mathbb{R}^{+}_m \times[0,T])$, so we have \eqref{j20-1}.

{\bf 2. }Similarly, recall \eqref{newprop:p2},  $\{\bar p^{\e} \}$ is uniformly (with respect to $\e$) bounded in $  L^\infty( \mathbb{R}^d_x \times V  \times[0,T] ) $ thus $w*$ precompact  so we have \eqref{j20-2}.

%{\color{red}

{\bf 3. }To prove \eqref{j20-5}, we have to control the tail for large $m$. We notice from the assumption \eqref{as:F1} that
$$
F(m, S^\e) < 0 \quad \text{when}\quad m>m_+,
$$ 
where $m_+$ is a uniform upper bound of $m_0(S^\e)$.  Now we take a smooth non-decreasing  function $\varphi(m) $ such that 
$$\varphi(m) = 0\quad\text{when}\quad  m<m_+  , \quad     \varphi(m) \leq m\quad\text{when}\quad  m_+ \leq m\leq 2m_+,  \quad     \varphi(m) = m\quad\text{when}\quad  m>2m_+, $$
 multiply with the equation \eqref{newepsilon}$_1$, and integrate with respect to $x,v,m$, then we have
$$
{d\over dt} \int_{m_+}^{\infty}\int_{\mathbb{R}^d_x} \int_V  \varphi(m) p^\e dv dx dm = 
 \int_{m_+}^{\infty}\int_{\mathbb{R}^d_x} \int_V \varphi'(m)  {F(m,S^\e)\over \e}      p^\e dv dx dm < 0, 
$$
thus
$$
\int_{2m_+}^{\infty}\int_{\mathbb{R}^d_x} \int_V m  p^\e dv dx dm \leq
\int_{m_+}^{\infty}\int_{\mathbb{R}^d_x} \int_V  \varphi(m) p^\e dv dx dm 
\leq \int_{m_+}^{\infty}\int_{\mathbb{R}^d_x} \int_V  \varphi(m) p_0 dv dx dm 
$$
\beq\label{f5-1}
\leq \int_{0}^{\infty}\int_{\mathbb{R}^d_x} \int_V  m p_0 dv dx dm  < \infty.
\eeq
The last inequality is based on the assumption \eqref{mar11-4} and gives the control at $\infty$. Now choose a test function $\phi(x,v,t) \in C_0 $ and a smooth cutoff function 
$$
0\leq \chi_R(m)\leq 1, \quad \chi_R(m) =1  \quad \text{for}\quad m<R, \quad \chi_R(m) =0  \quad \text{for}\quad m>2R. 
$$
Then we compute 
\beq\label{f6-1}
\begin{array}{rcl}
{\dis  \int_{0}^T \int_{\mathbb{R}^d_x} \int_V   \phi(x,v,t)  \bar p^\e dv dx dt }
&=& {\dis   \int_{0}^T \int_{\mathbb{R}^d_x} \int_V     \int_{0}^{\infty} \phi(x,v,t)    p^\e dm dv dx dt   } \vspace{2mm} \\
&=&{\dis   \int_{0}^T \int_{\mathbb{R}^d_x} \int_V     \int_{0}^{\infty} \phi(x,v,t)  \chi_R(m)  p^\e dm dv dx dt }\vspace{2mm} \\
& & + {\dis   \int_{0}^T \int_{\mathbb{R}^d_x} \int_V     \int_{0}^{\infty} \phi(x,v,t) (1- \chi_R(m))  p^\e dm dv dx dt .}
\end{array}
\eeq
For the first term on the right hand side, note from \eqref{j20-1} that $p^\e \rightharpoonup p   $ weak*, for any given $R$, 
$$
 \int_{0}^T \int_{\mathbb{R}^d_x} \int_V     \int_{0}^{\infty} \phi(x,v,t)  \chi_R(m)  p^\e dm dv dx dt
 \rightarrow   \int_{0}^T \int_{\mathbb{R}^d_x} \int_V     \int_{0}^{\infty} \phi(x,v,t)  \chi_R(m)  p  dm dv dx dt 
 \quad \text{as} \quad \e \rightarrow 0.
$$
For the second term, the control at $\infty$ estimate \eqref{f5-1} ensures 
$$
\int_{0}^T \int_{\mathbb{R}^d_x} \int_V     \int_{0}^{\infty} \phi(x,v,t) (1- \chi_R(m))  p^\e dm dv dx dt
\leq \| \phi \|_{C_0} \int_{0}^T \int_{\mathbb{R}^d_x} \int_V     \int_{R}^{\infty}   p^\e dm dv dx dt
$$
$$
\leq   {1\over R}  \| \phi \|_{C_0} \int_{0}^T \int_{\mathbb{R}^d_x} \int_V  \int_{R}^{\infty}  m p^\e dm dv dx dt  
\rightarrow 0 \text{ as } R\rightarrow \infty,
$$
 thus let $R\rightarrow \infty$ in \eqref{f6-1} we have
 $$
\int_{0}^T \int_{\mathbb{R}^d_x} \int_V   \phi(x,v,t)  \bar p^\e dv dx dt 
 \rightarrow   \int_{0}^T \int_{\mathbb{R}^d_x} \int_V     \int_{0}^{\infty} \phi(x,v,t)    p  dm dv dx dt \quad  \text{as}\quad \e \rightarrow 0.
 $$
Since we already have \eqref{j20-2}, by uniqueness of the limit we proved  \eqref{j20-5} in the sense of measure.
%$L^\infty ( \mathbb{R}^d_x \times V \times  [0,T] )$ -w*.

%}

{\bf 4. }For $\{S^\e (x, t) \}$, recall \eqref{newprop:ps}, it is bounded in 
$L^{1}\cap L^\infty( \mathbb{R}^d_x\times  [0,T] )$ uniformly with respect to $\e$, to show strong compactness, we use the claims
\begin{itemize}
\item 
Local compactness in space:   
\beq\label{f3-2}
S^\e(\cdot, t) \in W^{2,q} (\mathbb{R}^{d}_{x}), \quad \forall t\in [0,T],\quad \forall 1<q<\infty ,
\eeq
\item
 Local compactness in time:   
 \beq\label{f3-3}
\p_{t}S^\e(\cdot, t) \in W^{1,q} (\mathbb{R}^{d}_{x}), \quad \forall t\in [0,T],\quad \forall 1<q<\infty ,
\eeq
 \item
Control at $\infty$:  
\beq\label{j22-1}
 \int_{\mathbb{R}^{d}_{x}} \langle x \rangle (S^\e)^qdx  < \infty, \  \forall t\in[0,T],\quad \forall 1\leq q<\infty .
\eeq

\end{itemize}
The first two claims are elliptic estimates which can be similarly derived as in last section. We only prove \eqref{j22-1}.
Recall the equation for $p^\e$:
$$
 \p_t p^{\varepsilon} +  v \cdot \nabla_x p^{\varepsilon} + \p_m [  \f{F(m, S^{\varepsilon})}{\e}  p^{\varepsilon}] = {\dis \int}_V T(v,v',m) p^{\varepsilon}(x,v',m,t) - T(v',v,m) p^{\varepsilon}(x,v,m,t) dv' . 
$$
Multiply both sides of the equation by $\langle x\rangle$ and integrate with respect to $x,v,m$, we have
$$
{d\over dt} \int_{\mathbb{R}^d_x} \int_{0}^{\infty}\int_V  \langle x\rangle p^{\varepsilon} dv dm dx 
 = \int_{\mathbb{R}^d_x} \int_{0}^{\infty}\int_V  {v\cdot x \over \langle x\rangle }p^{\varepsilon} dv dm dx 
\leq  \|p_{0}\|_{L^1},
$$
that is,
$$
{d\over dt} \int_{\mathbb{R}^d_x} \langle x\rangle n^\e dx \leq  \|p_{0}\|_{L^1},
$$
thus we obtain
\beq\label{j22-2}
 \int_{\mathbb{R}^d_x} \langle x\rangle n^\e dx \quad\text{is uniformly bounded for all}\quad t\in [0,T].
\eeq

Next, recall the equation for $S^\e$:
$$
-\Delta S^\e +S^\e = n^\e, 
$$
multiply both sides of the equation by $\langle x\rangle (S^\e)^{q-1}$, for any $q$ large and fixed,  then integrate with respect to $x$, we have
$$
\int_{\mathbb{R}^d_x} \langle x\rangle (S^\e)^{q-1} (-\Delta S^\e +S^\e ) dx
=\int_{\mathbb{R}^d_x} \langle x\rangle (S^\e)^{q-1} n^\e dx,
$$
then
$$
\int_{\mathbb{R}^d_x} \langle x\rangle (S^\e)^q  dx
+\int_{\mathbb{R}^d_x} \nabla ( \langle x\rangle (S^\e)^{q-1}) \cdot \nabla S^\e dx
 = \int_{\mathbb{R}^d_x} \langle x\rangle (S^\e)^{q-1} n^\e dx.
$$
Note here that 
$$
\begin{array}{rcl}
\dis{ \int_{\mathbb{R}^d_x} \nabla ( \langle x\rangle (S^\e)^{q-1}) \cdot \nabla S^\e  dx }
&=& \dis{  \int_{\mathbb{R}^d_x} \langle x\rangle (p-1) (S^\e )^{q-2} \nabla S^\e  \cdot \nabla S^\e 
+ \int_{\mathbb{R}^d_x} (S^\e)^{q-1}  \nabla S^\e  \cdot \nabla\langle x\rangle  }
\vspace{2.5mm}  \\ 
&=&  \dis{ \int_{\mathbb{R}^d_x} \langle x\rangle (p-1) (S^\e )^{q-2} | \nabla S^\e  |^2
-  \int_{\mathbb{R}^d_x} {1\over q} (S^\e)^q \Delta\langle x\rangle }
\vspace{2.5mm}  \\
&=& \dis{  \int_{\mathbb{R}^d_x} \langle x\rangle (p-1) (S^\e )^{q-2} | \nabla S^\e  |^2
-  \int_{\mathbb{R}^d_x} {1\over q} (S^\e)^q 
({d-1\over \langle x\rangle} + {1\over \langle x\rangle^3}). }
\end{array}
$$
Then, we find 
$$
\int_{\mathbb{R}^d_x} \langle x\rangle (S^\e)^q  dx
+\int_{\mathbb{R}^d_x} \langle x\rangle (q-1) (S^\e )^{q-2} | \nabla S^\e  |^2dx
=  \int_{\mathbb{R}^d_x} {1\over q} (S^\e)^q 
({d-1\over \langle x\rangle} + {1\over \langle x\rangle^3})dx
+ \int_{\mathbb{R}^d_x} \langle x\rangle (S^\e)^{q-1} n^\e dx,
$$
that is
$$
\int_{\mathbb{R}^d_x} \langle x\rangle (S^\e)^q  dx
\leq  {d\over q} \int_{\mathbb{R}^d_x}  (S^\e)^q dx
+  \|S^\e \|_{L^\infty_x}^{q-1}  \int_{\mathbb{R}^d_x} \langle x\rangle n^\e dx,
$$
using again \eqref{newprop:ps} that  $\{S^\e (\cdot, t) \}$  is uniformly bounded in $L^{1}_x\cap L^\infty_x $  for all $t\in [0,T]$, and the fact \eqref{j22-2},
we conclude, for any fixed $q$ large,
$$
 \int_{\mathbb{R}^{d}_{x}} \langle x \rangle (S^\e)^q dx  < \infty, \  \forall t\in[0,T].
$$
 Note also \eqref{j22-1} already holds for $q=1$, which is because we can use the similar argument as for \eqref{sch2}, therefore we have proved \eqref{j22-1} for all $1\leq q < \infty$.

In conclusion, by using again the strong compactness criterion in Brezis \cite{Brezis},
 we have \eqref{j20-3} from the above claims.

 {\bf 5. }To prove \eqref{f3-1},  we first estimate the control of $S^\e$ at $x=\infty$. For simplicity of notation, we abbreviate the time variable in the following.

  Fix $R>1$.  For $|x|\geq R$, we compute
$$
\bea
|x|^{1/d} S^\e(x) &= \dis{\int   |x|^{1/d} G(x-y) n^\e  (y) dy         }    \vspace{2.5mm}  \\ 
& \leq \dis{ \int   |x-y|^{1/d}    G(x-y) n^\e  (y) dy + \int   |y|^{1/d}  G(x-y) n^\e  (y) dy         }    \vspace{2.5mm}  \\ 
& =: I ^\e+ II^\e.
\eea  
$$

  % whose Fourier transform is ${1\over 1+|\xi|^2}$. 
For $I^\e$, note that $G$ is only singular   at  the origin, we  cutoff the singularity  and   estimate 
$$
\bea
I^\e = & \dis{ \int_{|x-y|<1}   |x-y|^{1/d}    G(x-y) n^\e  (y) dy } + \dis{ \int_{|x-y|>1}   |x-y|^{1/d}    G(x-y) n^\e  (y) dy }    \vspace{2.5mm}  \\ 
\leq & \dis{ \int_{|x-y|<1}      G(x-y) n^\e  (y) dy } + \dis{ \Big( \int_{|x-y|>1}   |x-y|^{2/d}  G^2(x-y)  dy \Big)^{1\over 2}  \   \| n^\e \|_{L^2_x}  }     \vspace{2.5mm}  \\ 
\leq &   \| n^\e \|_{L^\infty_x} + \dis{ \Big( \int_{|x|>1}   |x|^{2/d}  G^2(x)  dx \Big)^{1\over 2}  \   \| n^\e \|_{L^2_x}  }   ,
\eea  
$$
recall  \eqref{newnlinf}, note also the $L^1$ bound of $\{n^\e\}$ is trivial from conservation of mass,  thus $\{n^\e\}$ is uniformly bounded in $L^1_x \cap L^\infty_x$,  and note the property (P3) of Bessel potential $G(x)$ in Proposition \ref{BP}, the integration above is finite. Then we conclude that $I^\e$ is uniformly bounded (the bound is independent of $\e$).

For $II^\e$,    we use H\"older's inequality, with $q={d \over d-1}$ and $q'=d$,  to get
\beq\label{mar11-1}
\bea
II^\e =  &  \dis{ \int   |y|^{1/d}  G(x-y) n^\e  (y) dy  \leq  \|G\|_{L^{d \over d-1}}   \Big(  \int   |y| n^\e(y)^d dy \Big)^{1\over d}}  \vspace{2.5mm}  \\ 
 \leq &  \dis{   \|G\|_{L^{d \over d-1}}      \Big(  \| n^\e \|_{L^\infty_x}^{d-1}   \int   |y| n^\e(y) dy \Big)^{1\over d} }.
\eea
\eeq
Here we use the property of Bessel potential  (P2) in  Proposition \ref{BP}, together with the estimates \eqref{newnlinf} and \eqref{j22-2}, we conclude from \eqref{mar11-1} that  $II^\e$ is also uniformly bounded (the bound is independent of $\e$).  Combine the above, we conclude
\beq\label{f13-2}
|x|^{1/d} S^\e(x) < \infty \quad \text{for all}\quad |x|>R  \   (>1).
 \eeq

       Next, by Morrey type Sobolev embedding Theorem (see Evans \cite{Evans}), note the elliptic estimates \eqref{f3-2}-\eqref{f3-3} imply that  the convergence of $ S^\e (x, t)   \rightarrow   S(x, t) $  is actually in H\"older space $ C^{0,1-\delta}( \mathbb{R}^d_x \times  [0,T] )  $, for any $\delta\in(0,1)$. Together with the estimates  \eqref{f13-2}, which yields that $S(x)$ vanishes at infinity, then  \eqref{f3-1} is proved. 
 $\hfill\square$

\end{proof}

\section{Conclusion}
 
We have considered a nonlinear kinetic chemotaxis model  with internal dynamics  incorporating signal transduction and adaptation.  Under some quite general assumptions on the model and the initial data, we have proved the  global existence of weak solution by using the Schauder fixed point theorem.  More precisely, for our mathematical treatment, we generalise the assumptions on the adaptation rate $F$ and turning kernel $T$ in the model than that in \cite{EH,EO,EO1}. Compare with the global existence  result in \cite{EH}, our result holds for any physical space dimensions. Moreover, the uniqueness of weak solution is also derived, based on some further regularity estimates on the solutions, following the method devised in \cite{LP}.

 Next, we considered the fast adaptation limit of this model  to Othmer-Dunbar-Alt type kinetic chemotaxis model. This limit  gives some  insight to the molecular origin of  the chemotaxis behaviour, by incorporating information about microscopic intracellular processes such as signal transduction and response  into the chemotaxis description.  This was done in \cite{EO,EO1} for a highly simplified description of intracellular dynamics, where  linear dynamics for the response to an extracellular signal was assumed. We remark also that  in order to derive the molecular origin of the chemotaxis behaviour, we did not use any moment closure to derive the closed evolution equation for the macroscopic density of cells as in \cite{EO,EO1}, instead, a kinetic type limit equation  \eqref{frl13} is arrived, with turning kernel incorporating information from microscopic intracellular processes.

In our analysis, the fast adaptation limit is derived by extracting a weak convergence subsequence in measure space. For this limit, the first difficulty is to show the concentration effect on  the internal state. 
When the small parameter $\e$, the adaptation time scale, goes to zero, we prove that the solution converges to a Dirac mass in the internal state variable. Another difficulty is the strong compactness argument on the chemical potential, which is essential for passing the nonlinear kinetic equation to the weak limit.

For future works, it is interesting to consider the case with several chemical reactions as related to the model proposed in Erban-Othmer \cite{EO,EO1}.  We also mention that other types of scaling and related limits can be considered.  For example, based on both fast adaptation and stiff response, the paper \cite{PTV} studied how the path-wise gradient of chemotactic signal arises from intra-cellular molecular content. 
Several other rescallings are possible. For instance, an open problem is the asymptotic behaviour of  the hyperbolic scaling, with \eqref{intr1}$_1$ replaced by
 $$
   \p_t p  +  v \cdot \nabla_x p  + \p_m [ { F(m, S ) \over \e}  p ] = {\dis {1\over \e} \int_V }T(v,v',m) p (x,v',m,t) - T(v',v,m) p (x,v,m,t) dv'.
   $$  
The hyperbolic limit is quite different from the fast adaptation limit, because, in that case, even the a priori estimate for $\bar p$ is not uniform in $\e$, so the limit equation is unclear, a difficult part of that will be to determine how the extracellular signal feeds into the chemotaxis response of the cells.

%Finally we mention that the first systematic derivation of a chemotaxis equation from a velocity jump process is due to Patlak \cite{Patlak}, who considers both internal and external biases in detail, but these biases are imposed. A basic assumption in \cite{Patlak} is that the run length is chosen and fixed whenever the particle turns As understood by further research, this process would be formally equivalent to a space jump process, see \cite{ODA} and references therein.  For other works  relating a process similar to the one treated here without the internal dynamics since Patlak's work, and the reader is referred to \cite{OH} for a review of the literature.

 % As was observed  elsewhere [30], the particle motion between turns is deterministic, and thus, were the speed and run length constant,
% In general one can show that this process leads to a renewal equation that generalizes the renewal equation (15) derived in \cite{ODA}, from which a diffusion equation is obtained by suitable choice of the waiting time and jump distributions. 

\section*{Acknowledgements.}  This work is mentored by  Professor Beno\^ \i t Perthame during the author's visit to Laboratoire Jacques-Louis Lions, UPMC, France.  Professor Beno\^ \i t Perthame provided  lots of fruitful suggestions and essential help.   The author is  deeply indebted to  him for his hospitality and generous help.  The research is partially support  by  National Natural Science Foundation of China (No. 11301182), and Science and Technology commission of Shanghai Municipality (No. 13ZR1453400).

%
%%%%%%%%%%%%%%%%%%%%%%%%%%%%%%%%%%%
%
%%%%%% BIBLIO %%%%%%%%%%%%%%%%%%%%%%
%
%%%%%%%%%%%%%%%%%%%%%%%%%%%%%%%%%%%%
%\pagestyle{myheadings}

\end{document}